\documentclass[11pt,leqno]{amsart}
\usepackage{amscd}
\usepackage[top=2.5 cm,bottom=2 cm,left=2.5 cm,right=2 cm]{geometry}
\usepackage[utf8]{inputenc}
\usepackage{bbm}
\usepackage{esint}
\usepackage{graphicx}
\usepackage{hyperref}
\usepackage{color}
\usepackage{amssymb}
\usepackage{amsfonts}
\usepackage{psfrag}
\newcounter{stepnb}

\usepackage{tikz}
\usetikzlibrary{shapes.geometric,calc,decorations.markings,math}

\tikzstyle{nodo}=[circle,draw,fill,inner sep=0pt,minimum size=%
1.5mm]
\tikzstyle{infinito}=[circle,inner sep=0pt,minimum size=0mm]

\newtheorem{thm}{Theorem}[section]
\newtheorem{lemma}[thm]{Lemma}

\numberwithin{equation}{section}

\newcommand{\be}{\begin{equation}}
\newcommand{\eq}{\end{equation}}

\begin{document}
\title[]{An overview of the local limit of non-local conservation laws, and a new proof of a compactness estimate
}

\author[M.~Colombo]{Maria Colombo}
\address{M.C. EPFL B, Station 8, CH-1015 Lausanne, Switzerland.}
\email{maria.colombo@epfl.ch}
\author[G.~Crippa]{Gianluca Crippa}
\address{G.C. Departement Mathematik und Informatik,
Universit\"at Basel, Spiegelgasse 1, CH-4051 Basel, Switzerland.}
\email{gianluca.crippa@unibas.ch}
\author[E.~Marconi]{Elio Marconi}
\address{E.M. Dipartimento di Matematica ``Tullio Levi-Civita", Universit\`a degli Studi di Padova, Via Trieste 63,
35131 Padua, Italy}
\email{elio.marconi@math.unipd.it}
\author[L.~V.~Spinolo]{Laura V.~Spinolo}
\address{L.V.S. CNR-IMATI ``E. Magenes", via Ferrata 5, I-27100 Pavia, Italy.}
\email{spinolo@imati.cnr.it}
\maketitle

\begin{footnotesize}
 Consider a non-local (i.e., involving a convolution term) conservation law: when the convolution term converges to a Dirac delta, in the limit we formally recover a classical (or ``local") conservation law. In this note we overview recent progress on this so-called non-local to local limit and in particular we discuss the case of anistropic kernels, which is extremely relevant in view of applications to traffic models. We also provide a new proof of a related compactness estimate. 
\medskip



\end{footnotesize}

\section{Introduction}

% Note the reference to cited works below.  Givental's paper in biblio.bib
% is labeled G, and Liu-Lian-Yau's paper is labeled LLY.
In recent years, the analysis of non-local conservation laws in the form 
\begin{equation} \label{e:nonloc}
     \partial_t u + \partial_x [ V(u \ast \eta) u] = 0 
\end{equation}
has attracted considerable attention. In the previous expression, the unknown $u$ is real-valued and depends on the variables $(t, x) \in \mathbb R_+ \times \mathbb R$. The Lipschitz continuous function $V: \mathbb R \to \mathbb R$ is assigned and so is the convolution kernel $\eta: \mathbb R\to \mathbb R_+$. The symbol $\ast$ denotes the convolution with respect to the space variable only. The analysis of~\eqref{e:nonloc} and of related equations is motivated by several applications concerning sedimentation~\cite{Sedimentation}, supply chains~\cite{ColomboHertyMercier}, pedestrian~\cite{ColomboGaravelloMercier} and vehicular~\cite{Chiarello} traffic models, and others. 
Existence and uniqueness results have been obtained under suitable assumptions in various works, see for instance~\cite{BlandinGoatin,ColomboGaravelloMercier,CLM,KeimerPflug}. 

In the present contribution we focus on the so-called non-local to local limit problem, which we now introduce. Fix $\varepsilon>0$ and consider the family of Cauchy problems 
\begin{equation} \label{e:nlc}
     \left\{
     \begin{array}{ll}
     \partial_t u_\varepsilon + \partial_x [ V(u_\varepsilon \ast \eta_\varepsilon) u_\varepsilon ] = 0 \\
     u_\varepsilon (0, \cdot) = u_0 \\
     \end{array}
     \right. 
     \qquad 
     \text{with $\eta_\varepsilon (x): = \frac{1}{\varepsilon} \eta \left( \frac{x}{\varepsilon}\right),$} 
\end{equation}
where the initial datum $u_0: \mathbb R \to \mathbb R$ satisfies suitable assumptions discussed in the following.  In the non-local to local limit $\varepsilon \to 0^+$ the convolution kernel $\eta_\varepsilon \stackrel{*}{\rightharpoonup} \delta_{x=0}$ weakly-$^\ast$ in the sense of measures. Consequently, what one \emph{formally} recovers in the vanishing $\varepsilon$ limit is the scalar conservation law 
\begin{equation} \label{e:cl}
       \left\{
     \begin{array}{ll}
     \partial_t u + \partial_x [ V(u) u] = 0 \\
     u (0, \cdot) = u_0 . \\
     \end{array}
     \right. 
\end{equation}
Whether or not one can rigorously establish the convergence of  $u_\varepsilon$ to the entropy admissible solution of the conservation law~\eqref{e:cl} has been the target   
of recent investigations, that we briefly overview in the next section. We refer instead to the very classical references~\cite{Dafermos_book,Kruzkov} for the definition of entropy admissible solutions of a conservation law. 

This note is organized as follows: in \S\ref{s:nltl} we discuss recent results on the non-local to local limit. In \S\ref{s:heu} and \S\ref{s:complete} we provide an alternative proof of the main compactness estimate in~\cite{ColomboCrippaMarconiSpinolo2} concerning the non-local to local limit, which is estimate~\eqref{e:tv} below. The proof of~\eqref{e:tv} we discuss here is longer and less direct then the original one given in~\cite{ColomboCrippaMarconiSpinolo2}, but we think it is of interest as we feel it is more transparent as it more clearly elucidates the basic mechanism yielding compactness. In \S\ref{s:heu} we discuss an heuristic argument, whereas in \S\ref{s:complete} we provide the complete proof.

\section{The non-local to local limit: an overview} \label{s:nltl}
Let us consider the family of Cauchy problems~\eqref{e:nlc}: one of the main challenges in studying the vanishing $\varepsilon$ limit is that, owing to the presence of the non-local term, it is fairly hard to gain explicit insights on the precise behavior of the solution, and henceforth useful compactness estimates. Note furthermore that, if $u_0 \in L^1(\mathbb R)$, it is fairly easy to establish bounds on $\| u_\varepsilon (t, \cdot)\|_{L^1(\mathbb R)}$ that are uniform in $\varepsilon$ and $t$.  This in turn implies that, up to subsequences, the family $u_\varepsilon$ converges to some limit measure weakly$^\ast$ in the sense of measures on $[0, T] \times \mathbb R$, for every $T>0$. However, weak convergence alone does not suffice to pass to the limit in the nonlinear term $V(u_\varepsilon \ast \eta_\varepsilon) u_\varepsilon$ and henceforth to pass from~\eqref{e:nlc} to~\eqref{e:cl}.
  
To the best of our knowledge, one of the very first results concerning the vanishing $\varepsilon$ limit is the paper by Zumbrun~\cite{Zumbrun}, where, among other things, the author establishes convergence of $u_\varepsilon$ to the entropy admissible solution $u$ of~\eqref{e:cl} under the assumptions that $\eta$ is even, that is $\eta(x) = \eta (-x)$, and (quite restrictively)  that $u$ is smooth. 

The non-local to local limit problem is again addressed in the more recent paper~\cite{ACT} and investigated by relying on numerical experiments. The numerical simulations exhibited in~\cite{ACT} suggest that, as $\varepsilon$ gets closer and closer to $0$, the solution $u_\varepsilon$ of~\eqref{e:nlc} approaches the entropy admissible solution of~\eqref{e:cl}. In~\cite{ColomboCrippaSpinolo} the authors provide counter-examples showing that actually this is not always the case. More precisely, the 
counter-examples dictate that in general $u_\varepsilon$ does not converge to the entropy admissible solution, not even weakly or up to subsequences. This is very loosely speaking achieved by singling out a property (a different one in each counter-example) which is satisfied by the solution $u_\varepsilon$ of~\eqref{e:nlc} for every $\varepsilon$, passes to the limit, but is \emph{not} satisfied by the entropy admissible solution of~\eqref{e:cl}.  Another result in~\cite{ColomboCrippaSpinolo} focuses on the ``viscous counter-part" of the non-local to local limit problem. More precisely, fix $\nu>0$ and consider the family of viscous non-local problems 
\begin{equation} \label{e:nlcv}
     \left\{
     \begin{array}{ll}
     \partial_t u_{\varepsilon \nu} + \partial_x [ V(u_{\varepsilon \nu} \ast \eta_\varepsilon) u_{\varepsilon \nu} ] = \nu \partial_{xx} u_{\varepsilon \nu}\\
     u_{\varepsilon \nu} (0, \cdot) = u_0 \\
     \end{array}
     \right. 
     \qquad 
     \text{with $\eta_\varepsilon (x): = \frac{1}{\varepsilon} \eta \left( \frac{x}{\varepsilon}\right).$} 
\end{equation}
Existence and uniqueness results for~\eqref{e:nlcv} can be established by relying on fairly standard techniques, see \S2 in~\cite{ColomboCrippaSpinolo}. Extending previous results by Calderoni and Pulvirenti~\cite{CP}, Theorem 1.1 in~\cite{ColomboCrippaSpinolo} states that, under quite general assumptions on $V$, $\eta$ and $u_0$, the solutions $u_{\varepsilon \nu}$ converge as $\varepsilon \to 0^+$ to the solution of the viscous conservation law 
\begin{equation} \label{e:cv}
     \left\{
     \begin{array}{ll}
     \partial_t u_{\nu} + \partial_x [ V(u_{\nu}) u_{ \nu} ] = \nu \partial_{xx} u_{\nu}\\
     u_{\nu} (0, \cdot) = u_0. \\
     \end{array}
     \right. 
\end{equation}
This result is also relevant from the numerical viewpoint, see the related discussion in~\cite{ColomboCrippaGraffSpinolo}. 
Wrapping up, one has the following diagram:
\begin{equation}
\label{e:disegno}
\minCDarrowwidth70pt
\begin{CD}
 \partial_t u_{\varepsilon \nu} + \partial_x \big[ u_{\varepsilon \nu} V(u_{\varepsilon \nu} \ast \eta_\varepsilon) \big] = \nu \partial_{xx} u_{\varepsilon \nu}     @> \varepsilon \to 0^+  >  \text{\cite[Theorem 1.1]{ColomboCrippaSpinolo}}>  \partial_t u_\nu + \partial_x \big[ u_\nu V (u_\nu) \big] = \nu \partial_{xx} u_\nu \\
@V \nu \to 0^+  V \text{\cite[Proposition 1.2]{ColomboCrippaSpinolo}} 
V        @V \nu \to 0^+ V \text{Kru{\v{z}}kov's Theorem} V\\
 \partial_t u_{\varepsilon } + \partial_x \big[ u_{\varepsilon } V(u_{\varepsilon } \ast \eta_\varepsilon) \big] = 0      @>  \varepsilon \to 0^+ > \text{False in general}>   \partial_t u + \partial_x \big[ u V(u) \big] = 0
\end{CD}
\end{equation}
Note that the ``vertical" convergence at the left of the previous diagram follows from Proposition 1.2 in~\cite{ColomboCrippaSpinolo} and relies on the extension of classical parabolic compactenss estimates to the nonlocal setting, whereas the convergence of the solutions of the viscous conservation law to the entropy admissible solution of~\eqref{e:cl} dates back to the by-now classical work of Kru{\v{z}}kov~\cite{Kruzkov}. See also~\cite{CDNKP} for more recent results on the ``diagonal" convergence in the previous diagram. 

Going back to the original (non-viscous) non-local to local limit problem, the above mentioned counter-examples in~\cite{ColomboCrippaSpinolo} left open the possibility of establishing convergence in a more specific setting. As a matter of fact, in the last very few years several results have been obtained under assumptions that are fairly natural in view of applications to traffic flow models, which we now discuss. 

The archetype of fluido-dynamic traffic models is the celebrated LWR model, introduced in the works by Lighthill and Whitham~\cite{LW} and Richards~\cite{R}, which involves a conservation law like the one at the first line of~\eqref{e:cl}. In the LWR model the unknown $u$ represents the density of vehicles, and $V$ their speed. The model postulates that drivers tune their speed based on the pointwise traffic density, and, since the most common reaction to an increase in the car density is deceleration, the assumptions usually imposed on $V$ are 
\begin{equation} \label{e:V}
    V \in \mathrm{Lip}(\mathbb R), \qquad V' \leq 0.
\end{equation}
We now turn to the assumptions satisfied by the datum $u_0$: since it represents the initial density, one assumes $u_0 \ge 0$. The datum $u_0$ should also not exceed the maximum possible density (corresponding to bumper-to-bumper packing), which with no loss of generality we can assume normalized to $1$. Wrapping up, 
\begin{equation}
\label{e:u0}
     u_0 \in L^\infty (\mathbb R), \quad 0 \leq u_0 (x)\leq 1 \; \text{for a.e. $x \in \mathbb R$}. 
\end{equation}
Let us now turn to the non-local LWR model~\eqref{e:nonloc}: compared to the classical one, this version of the model aims at taking into account that drivers tune their speed according to the car density in a suitable  neighborhood of their position, rather than to the pointwise density only.  Also, it turns out that the presence of the convolution term in~\eqref{e:nonloc} rules out the possible presence of infinite acceleration, a well-known drawback of the classical LWR model.  The hypotheses most commonly imposed on $\eta$ in non-local LWR models are then 
\begin{equation} \label{e:eta}
      \eta \in L^1(\mathbb R) \cap L^\infty (\mathbb R), \quad  \eta \ge 0,  \quad \int_{\mathbb R} \eta(x) dx =1, \quad 
 \mathrm{supp} \, \eta \subseteq \mathbb R_-, \qquad  \eta \text{ non-decreasing on $\mathbb R_-$}. 
      \end{equation}
The first three conditions in~\eqref{e:eta} are fairly standard assumptions for convolution kernels. More interesting is the second-last condition, a ``look-ahead-only" assumption that takes into account that drivers tune their speed according to the downstream traffic density only. Finally, the last condition in~\eqref{e:eta} expresses the fact that drivers are more deeply influenced by closer vehicles rather than by those that are further away. 

The analysis on the non-local LWR model~\eqref{e:nlc} under~\eqref{e:V},\eqref{e:u0} and \eqref{e:eta} was initiated by Blandin and Goatin~\cite{BlandinGoatin}, who in particular established the maximum principle 
\begin{equation}
\label{e:mp2}
     0 \leq u_\varepsilon (t, x) \leq 1 \; \text{for a.e. $(t, x) \in \mathbb R_+ \times \mathbb R$},
\end{equation}
a remarkable property in view of applications that is not satisfied by general non-local conservation laws. Keimer and Pflug~\cite{KeimerPflug2}  were instead the first, to the best of our knowledge, to investigate, under the further assumption that the initial datum $u_0$ is a monotone function, the local limit of the non-local LWR model. 
In~\cite{BressanShen,BressanShen2} Bressan and Shen established convergence for general initial data of bounded total variation and bounded away from $0$. The analysis in~\cite{BressanShen,BressanShen2} relies on a change of variables that allows to rewrite the equation at the first line of~\eqref{e:nlc} as a system with relaxation, and requires that $\eta (x) = e^x \mathbbm{1}_{\mathbb R_-}(x)$, where $\mathbbm{1}_{\mathbb R_-}$ denotes the characteristic function of the negative real axis. 
In~\cite{ColomboCrippaMarconiSpinolo} we established convergence under quite general assumptions on the convolution kernel $\eta$, but requiring more restrictive assumptions than in~\cite{BressanShen,BressanShen2} on the initial datum.  

Without entering the details, an important technical point is that a key issue in 
all the above works~\cite{BressanShen,BressanShen2,ColomboCrippaMarconiSpinolo,KeimerPflug2} is showing that, under the different assumptions considered in each paper, the total variation $\mathrm{TotVar} \, u_\varepsilon (t, \cdot)$ is uniformly bounded in $\varepsilon$ and $t$. Owing to the Helly-Kolmogorov-Fr\'echet Theorem, this yields compactness of the family $\{ u_\varepsilon (t, \cdot)\}$ (some more work is then required to show that $u_\varepsilon$ converges to the entropy admissible solution of~\eqref{e:cl}). As a matter of fact, the prevailing feeling in the nonlocal conservation laws community was that the assumptions~\eqref{e:V},\eqref{e:u0} and~\eqref{e:eta} should suffice to establish a uniform bound on $\mathrm{TotVar} \, u_\varepsilon (t, \cdot)$: this was also supported by numerical evidence, see~\cite{ACT,BlandinGoatin}. However, in~\cite{ColomboCrippaMarconiSpinolo} we exhibit a counter-example showing that actually this is not always the case: in~\cite[\S4]{ColomboCrippaMarconiSpinolo} we construct an initial datum $u_0$ such that $\mathrm{TotVar} \, u_0 < + \infty$, but the solution of~\eqref{e:nlc} satisfies 
\begin{equation} \label{e:tvblowup}
     \sup_{\varepsilon>0} \mathrm{TotVar} \, u_\varepsilon (t, \cdot) = + \infty, \quad \text{for every $t>0$.}
\end{equation}
This happens when $V(u) =1-u$ and for a fairly large class of convolution kernels satisfying~\eqref{e:eta}. 
In particular, the choice $\eta (x) = e^x \mathbbm{1}_{\mathbb R_-}(x)$ is possible, and the reason why the counter-example does not contradict the uniform bounds established in~\cite{BressanShen,ColomboCrippaMarconiSpinolo} is because 
the initial datum $u_0$ triggering the total variation blow-up attains the value $0$, which is not allowed in~\cite{BressanShen,ColomboCrippaMarconiSpinolo}. And indeed, by looking at the construction of the counter-example, one realizes that the fact that $u_0$ vanishes on suitable intervals is essential in the proof of~\eqref{e:tvblowup}. 

We remark in passing that the maximum principle~\eqref{e:mp2} yields compactnes in the $L^\infty$ weak-$\ast$
topology, but, as pointed out before, owing to non-linearity weak convergence alone does not suffice to pass to the limit in~\eqref{e:nlc} to get~\eqref{e:cl}. Wrapping up, the total variation blow up~\eqref{e:tvblowup} apparently sets severe constraints on the possibility of establishing convergence in the non-local to local limit. An elegant way out this obstruction was found in the paper~\cite{CocliteCoronDNKeimerPflug}: rather than $u_\varepsilon$, the authors consider the convolution term 
\begin{equation}  \label{e:w}
      w_\varepsilon (t, x) : = u_\varepsilon \ast \eta_\varepsilon (t, x) = \frac{1}{\varepsilon} \int_{x}^{+\infty} \eta \left( \frac{x-y}{\varepsilon} \right) u_\varepsilon (t, y) dy 
\end{equation}
and prove that, under suitable assumptions, the total variation $\mathrm{TotVar} \, w_\varepsilon (t, \cdot)$ is a monotone 
non-increasing function of time, and as such uniformly bounded in $\varepsilon$ and $t>0$. In other words, the total variation of $u_\varepsilon (t, \cdot)$ may blow up, but considering the convolution term allows to gain a bit of regularity sufficient to establish total variation bounds. As a drawback, the analysis in~\cite{CocliteCoronDNKeimerPflug}, as the one in~\cite{BressanShen,BressanShen2},  is restricted to the case $\eta (x) = e^x \mathbbm{1}_{\mathbb R_-}(x)$, which entails the algebraic identity
$$
      u_\varepsilon = w_\varepsilon - \varepsilon \partial_x w_\varepsilon. 
$$
The above identity, in turn, allows the authors of~\cite{CocliteCoronDNKeimerPflug} to find an equation for $w_\varepsilon$ which does not contain terms involving $u_\varepsilon$ (every time the authors have one such term, they replace it by using the above identity), and this is crucially used in the proof of the total variation bound. The convergence argument is then quickly concluded by relying on the argument in~\cite{BressanShen,BressanShen2}. See also~\cite{OttoAutori} for other recent results on the non-local to local limit in the case of exponential kernels.  

In the recent paper~\cite{ColomboCrippaMarconiSpinolo2} we extend the analysis in~\cite{CocliteCoronDNKeimerPflug} by removing the assumption that $\eta (x) = e^x \mathbbm{1}_{\mathbb R_-}(x)$. 
More precisely, in~\cite{ColomboCrippaMarconiSpinolo2} we only impose conditions~\eqref{e:eta}, entirely natural in view of applications to traffic models, plus the convexity assumption 
\begin{equation}
\label{eqn:convex}
\eta \mbox{ is convex on }  \mathbb R_-,
\end{equation}
that we discuss in the following. We now quote \cite[Theorem 1.1]{ColomboCrippaMarconiSpinolo2}.
\begin{thm}\label{t:main}Assume that $u_0$, $V$ and $\eta$ satisfy~\eqref{e:V},~\eqref{e:u0}, and~\eqref{e:eta},~\eqref{eqn:convex}, respectively. Assume furthermore that 
$\mathrm{TotVar} \, u_0 < + \infty$. Then 
\begin{equation} \label{e:tv}
     \mathrm{TotVar} \, w_\varepsilon(t, \cdot) \leq \mathrm{TotVar} \, w_\varepsilon(0, \cdot)  \quad \text{for every $\varepsilon>0$ and a.e. $t>0$},
\end{equation}
where $w_\varepsilon$ is the same as in~\eqref{e:w}. 
\end{thm}
Once more, we remark that all assumptions~\eqref{e:V},~\eqref{e:u0}, and~\eqref{e:eta} are entirely consistent with the traffic models framework. In the next two sections we provide a proof of Theorem~\ref{t:main} alternative to the original one in~\cite{ColomboCrippaMarconiSpinolo2}. For the time being we point out that by relying on~\eqref{e:tv} 
and some further elementary arguments it is fairly easy to see that, up to subsequences, 
\begin{equation}
\label{e:converge}
        w_\varepsilon \to u \; \text{strongly in $L^1_\mathrm{loc} (\mathbb R_+ \times \mathbb R)$}, 
       \qquad  u_\varepsilon \stackrel{*}{\rightharpoonup} u \; \text{weakly$^\ast$ in $L^\infty (\mathbb R_+ \times \mathbb R)$},
\end{equation}
where $u$ is a \emph{distributional} solution of the Cauchy problem~\eqref{e:cl}.  What a priori is not at all clear is that 
$u$ is the \emph{entropy admissible} solution of~\eqref{e:cl}: to get this, we introduce a new and fairly general entropy admissibility criterion for non-local to local limits, see~\cite[Theorem 1.2]{ColomboCrippaMarconiSpinolo2}, which eventually concludes the convergence proof. By relying on an argument due to Kuznetsov~\cite{Kuznetsov} we also get a convergence rate, see~\cite[Theorem 1.3]{ColomboCrippaMarconiSpinolo2}. What is left to be discussed is~\eqref{eqn:convex}, the only one among of our assumptions that is not entirely natural in view of the applications to traffic models: for instance, a convolution kernel commonly used in applications is $\eta (x) = \mathbbm{1}_{[-1, 0]}$, which violates~\eqref{eqn:convex}. It turns out that, if $\eta$ does not satisfy~\eqref{eqn:convex}, then our main compactness estimate~\eqref{e:tv} fails in general: an explicit counter-example is constructed in~\cite[\S6]{ColomboCrippaMarconiSpinolo2}. 

To conclude our overview, we touch upon a recent work concerning a  problem closely related to~\eqref{e:nlc}. In~\cite{FGKP} the authors consider the ``non-local in velocity" problem
\begin{equation} \label{e:nlv}
     \left\{
     \begin{array}{ll}
     \partial_t u_\varepsilon + \partial_x [ (V(u_\varepsilon) \ast \eta_\varepsilon) u_\varepsilon ] = 0 \\
     u_\varepsilon (0, \cdot) = u_0 \\
     \end{array}
     \right. 
     \qquad 
     \text{with $\eta_\varepsilon (x): = \frac{1}{\varepsilon} \eta \left( \frac{x}{\varepsilon}\right),$} 
\end{equation} 
where the convolution is applied to the velocity function directly, rather than on the solution $u_\varepsilon$. Note that~\eqref{e:nlv} coincides with~\eqref{e:nlc} in the linear case $V(u) = 1-u$, but differs in general. In~\cite{FGKP} the authors show that the solution $u_\varepsilon$ of~\eqref{e:nlc} converge to the entropy admissible solution of~\eqref{e:nlc} if either $u_0$ is monotone or $\eta (x) = e^x \mathbbm{1}_{\mathbb R_-}(x)$. See also~\cite{CKR,DHSS,KPpreprint} for other recent work concerning problems related to~\eqref{e:nlc}.

\section{Proof of Theorem~\ref{t:main}: heuristic argument}\label{s:heu}
In this section we discuss an handwaving argument for the total variation estimate~\eqref{e:tv}, which we hope provides the main idea underpinning the rigorous argument. The complete proof is then established in the next section. 
\subsection{Preliminary results}
We briefly recall some known results (see for instance~\cite{ColomboCrippaMarconiSpinolo2}) that we need in the following. First, owing to~\eqref{e:mp2} we have 
\begin{equation}\label{e:mp}
   0 \leq w_\varepsilon (t, x) \leq 1  \; \text{for a.e. $(t, x) \in \mathbb R_+ \times \mathbb R$}.
\end{equation}
Second, $w_\varepsilon$ is a Lipschitz continuous function, $w_\varepsilon \in W^{1, \infty} (\mathbb R_+ \times \mathbb R)$. 
To fix the notation, we also introduce the characteristic line $X_\varepsilon(\cdot, x, s)$, which is the solution of the Cauchy problem 
\begin{equation} \label{e:X}
\left\{
\begin{array}{ll}
    \displaystyle{\frac{d  X_\varepsilon }{dt} = V(w_\varepsilon (t, X_\varepsilon))}\\
    \phantom{cc} \\
    X_\varepsilon (s, x, s) = x. \\
\end{array}
\right.
\end{equation}
A straightforward computation yields the expression for the material derivative of $w_\varepsilon$, namely 
\begin{equation}
\label{e:ewgen}
      \partial_t w_\varepsilon + V(w_\varepsilon) \partial_x w_\varepsilon = 
      \frac{1}{\varepsilon^2} \int_x^{+\infty} \eta' \left( \frac{x-y}{\varepsilon} \right) [V(w_\varepsilon (t, x)) - V(w_\varepsilon (t, y) )] u_\varepsilon (t, y)dy. 
\end{equation}
Also, one can also show that the map $g: \mathbb R_+ \to \mathbb R$ define by setting 
\begin{equation} \label{e:tvmap}
    g(t) : = \mathrm{TotVar} \, w_\varepsilon (t, \cdot)
\end{equation}
is Lipschitz continuous. 
To conclude, we state a known and elementary result on Lipschitz continuous functions that we need in the following, and provide the proof for the sake of completeness. 
\begin{lemma} \label{l:fg}
Let $I \subseteq \mathbb R$ be an open interval and $f, g : I \to \mathbb R$ two Lipschitz continuous functions  such
that $f \leq g$ on I. If $f$ and $g$ are both differentiable at the point $t_\ast \in I$ and $f(t_\ast) = g(t_\ast)$ then $f'(t_\ast) = g'(t_\ast)$.  
\end{lemma}
\begin{proof}
Up to replacing $f$ by $f-g$, we can assume with no loss of generality that $g=0$. Since $f \leq 0$ and $f(t_\ast)=0$ then $t_\ast$ is a point of maximum for $f$, which yields $f'(t_\ast) = 0$.
\end{proof}
\subsection{The heuristic argument}
Recall~\eqref{e:tvmap} and assume for a moment that we have shown that 
\begin{equation} \label{e:vagiu}
     g'(t_\ast)  \leq 0
\end{equation}
at every point $t_\ast$ at which the map~\eqref{e:tvmap} is differentiable. An integration in time then yields~\eqref{e:tv}. 

In our heuristic argument we establish~\eqref{e:vagiu} at every point $t_\ast$ of differentiability for $g$ that satisfies some further suitable conditions. More precisely, we assume that at time $t_\ast$ the map $w_\varepsilon (t_\ast, \cdot)$ is compactly supported and has a finite number of local maxima and local minima. To focus on the main issues and make the other details as simple as possible, we perform the explicit computations in the case where $w_\varepsilon (t_\ast, \cdot)$ has exactly two local maxima, which we term $x_1$ and $x_3$, and one local minimum, which we term 
$x_2$, satisfying 
\begin{equation}
\label{e:ordinamento}
    x_1 < x_2 < x_3. 
\end{equation}
See Figure~\ref{f} for a representation. The analysis extends to the case of finitely many extrema, at the price of some elementary but rather tedious computations  (see also the argument in \S\ref{s:complete}). 
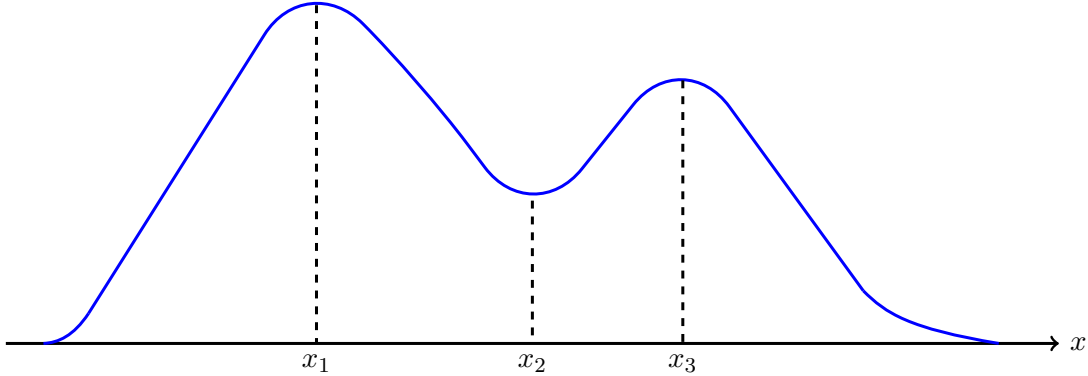
\begin{figure}
\begin{center}
\caption{In blue the function $w_\varepsilon (t_\ast, \cdot)$ in the heuristic argument.} 
\label{f}
%\bigskip
\begin{tikzpicture}
%\draw[line width=0.4mm,->] (0,0) -- (0,6) node[anchor=south] {$w_\varepsilon (t_\ast, \cdot)$};
\draw[line width=0.4mm, ->]   (-7, 0) -- (7, 0) node[anchor=west] {$x$}; 
 \draw[line width=0.4mm, blue] (-6.5,0) .. controls (-6.3, 0.01) and (-6.1, 0.1).. (-5.9, 0.4);
 \draw[line width=0.4mm, blue] [rounded corners=30pt] (-5.9, 0.4) -- (-3, 5) -- (-1.5, 3.5)-- (0, 1.5)--(2, 4)--(4.4, 0.7);
  \draw[line width=0.4mm, blue] (4.4,0.7) .. controls (4.7, 0.4) and (5, 0.2).. (6.2, 0);
\draw[dashed, line width=0.4mm]   (-2.87, 4.5) -- (-2.87, 0) node[anchor=north] {$x_1$};   
\draw[dashed, line width=0.4mm]   (0, 1.9) -- (0, 0) node[anchor=north] {$x_2$};  
\draw[dashed, line width=0.4mm]   (2, 3.5) -- (2, 0) node[anchor=north] {$x_3$};   
%\draw  (1,1) node {$\mf{\bar u}$};
\end{tikzpicture}
\end{center}
\end{figure}
We now set 
\begin{equation}\label{e:f}
    f(t) : = 2 \big[w_\varepsilon (t_\ast, X_\varepsilon (t, x_1, t_\ast)) - w_\varepsilon (t_\ast, X_\varepsilon(t, x_2, t_\ast))+  
        w_\varepsilon (t_\ast, X_\varepsilon (t, x_3, t_\ast)) \big],
\end{equation}
where $X_\varepsilon$ is the same as in~\eqref{e:X}, and recall~\eqref{e:tvmap}. Note that $f(t) \leq g(t)$ for every $t\ge 0$. Under our assumptions on $w_\varepsilon (t_\ast, \cdot)$ (see Figure~\ref{f}) we also have 
\begin{equation}
\label{e:sonouguali}
        g(t_\ast) =  \mathrm{TotVar} \, w_\varepsilon (t_\ast, \cdot) =   2 \big[ w_\varepsilon (t_\ast, x_1) - w_\varepsilon (t_\ast, x_2)+  
        w_\varepsilon (t_\ast, x_3) \big] = f(t_\ast).
\end{equation}
 By applying Lemma~\ref{l:fg} we conclude that to establish~\eqref{e:vagiu} it suffices to show that 
\begin{equation}\label{e:fd}
f'(t_\ast) \leq 0. 
\end{equation}
Towards this end, we recall the expression~\eqref{e:ewgen} for the material derivative, and conclude that 
\begin{equation*}
\begin{split}
f'(t_\ast)  = &  2 \left[ \frac{1}{\varepsilon^2} \int_{x_1}^{+\infty} \eta' \left( \frac{x_1-y}{\varepsilon} \right) [V(w_\varepsilon (t_\ast, x_1)) - V(w_\varepsilon (t_\ast, y) )] u_\varepsilon (t_\ast, y)dy \right. \\
&  - \frac{1}{\varepsilon^2} \int_{x_2}^{+\infty} \eta' \left( \frac{x_2-y}{\varepsilon} \right) [V(w_\varepsilon (t_\ast, x_2)) - V(w_\varepsilon (t_\ast, y) )] u_\varepsilon (t_\ast, y)dy \\
& + \left. \frac{1}{\varepsilon^2} \int_{x_3}^{+\infty} \eta' \left( \frac{x_3-y}{\varepsilon} \right) [V(w_\varepsilon (t_\ast, x_3)) - V(w_\varepsilon (t_\ast, y) )] u_\varepsilon (t_\ast, y)dy \right].
\end{split}
\end{equation*}
We now recall~\eqref{e:ordinamento} and rearrange the above terms as 
\begin{equation*}
\begin{split}
&f'(t_\ast)  =   2 \Bigg[ \frac{1}{\varepsilon^2}\underbrace{\int_{x_1}^{x_2} \eta' \left( \frac{x_1-y}{\varepsilon} \right) [V(w_\varepsilon (t_\ast, x_1)) - V(w_\varepsilon (t_\ast, y) )] u_\varepsilon (t_\ast, y)dy}_{:=I_1}  \\
& +  \frac{1}{\varepsilon^2} \underbrace{\int_{x_2}^{+ \infty} \eta' \left( \frac{x_1-y}{\varepsilon} \right) [V(w_\varepsilon (t_\ast, x_1)) - V(w_\varepsilon (t_\ast, x_2) )] u_\varepsilon (t_\ast, y)dy}_{:=I_2} \\
& + \frac{1}{\varepsilon^2} \underbrace{\int_{x_2}^{x_3} \left[ \eta' \left( \frac{x_1-y}{\varepsilon} \right) -\eta' \left( \frac{x_2-y}{\varepsilon} \right)
        \right] [V(w_\varepsilon (t_\ast, x_2)) - V(w_\varepsilon (t_\ast, y) )] u_\varepsilon (t_\ast, y)dy}_{:=I_3}  \\
&  + \frac{1}{\varepsilon^2} \underbrace{\int_{x_3}^{+ \infty} \left[ \eta' \left( \frac{x_1-y}{\varepsilon} \right) -\eta' \left( \frac{x_2-y}{\varepsilon} \right)
        \right] [V(w_\varepsilon (t_\ast, x_2)) - V(w_\varepsilon (t_\ast, x_3) )] u_\varepsilon (t_\ast, y)dy}_{:=I_4}  \\
& + \! \frac{1}{\varepsilon^2} \! \underbrace{\int_{x_3}^{+ \infty} \! \!  \left[ \!  \eta' \left(\!  \frac{x_1-y}{\varepsilon} \! \right)\! \!   -\! \eta' \left(\!  \frac{x_2-y}{\varepsilon} \! \right)
         \! + \!  \eta' \left( \!  \frac{x_3-y}{\varepsilon} \! \right) 
        \! \right] [V(w_\varepsilon (t_\ast, x_3))\!  -\!  V(w_\varepsilon (t_\ast, y) )] u_\varepsilon (t_\ast, y)dy}_{:=I_5} \! \Bigg].       
\end{split}
\end{equation*}
To control $I_1$ and $I_2$, we recall that  by assumption (see Figure~\ref{f})
\begin{equation*}
\begin{split}
   w_\varepsilon  (t_\ast, x_1) \ge w_\varepsilon (t_\ast, y)&  \; \text{for every $y \in [x_1, x_2]$} \\
   & \stackrel{V'\leq 0}{\implies} V(w_\varepsilon (t_\ast, x_1)) \leq  V(w_\varepsilon (t_\ast, y) ) \; \text{for every $y \in [x_1, x_2]$}. 
\end{split}
\end{equation*}
Since $\eta' \ge 0$ owing to~\eqref{e:eta} and $u_\varepsilon \ge 0$ by~\eqref{e:mp}, this yields $I_1 \leq 0$, $I_2 \leq 0$. 
To control $I_3$ and $I_4$, we point out that 
\begin{equation*}
\begin{split}
    w_\varepsilon (t_\ast, x_2) \leq w_\varepsilon (t_\ast, y) & \; \text{for every $y \in [x_2, x_3]$} \\ & 
   \stackrel{V'\leq 0}{\implies} V(w_\varepsilon (t_\ast, x_2)) \ge  V(w_\varepsilon (t_\ast, y) ) \; \text{for every $y \in [x_2, x_3]$}. 
\end{split}
\end{equation*}
Since by assumption~\eqref{eqn:convex} the function $\eta'$ is monotone non-decreasing on $\mathbb R_-$, we have 
$$
    \eta' \left( \frac{x_1-y}{\varepsilon} \right) \leq \eta' \left( \frac{x_2-y}{\varepsilon} \right)
$$   
and hence $I_3 \leq 0$, $I_4 \leq 0$.  We are left to control $I_5$. By arguing as before, we have 
\begin{equation*}
\begin{split}
    w_\varepsilon (t_\ast, x_3) \ge w_\varepsilon (t_\ast, y) & \; \text{for every $y \in [x_3, + \infty[$} \\ &
   \stackrel{V'\leq 0}{\implies} V(w_\varepsilon (t_\ast, x_3)) \leq  V(w_\varepsilon (t_\ast, y) ) \; \text{for every $y \in [x_3, + \infty[$} 
\end{split}
\end{equation*}
and 
$$
    \eta' \left( \frac{x_1-y}{\varepsilon} \right) -\eta' \left( \frac{x_2-y}{\varepsilon} \right)
         +  \eta' \left( \frac{x_3-y}{\varepsilon} \right)\stackrel{\eta'\ge 0}{\ge} -\eta' \left( \frac{x_2-y}{\varepsilon} \right)
         +  \eta' \left( \frac{x_3-y}{\varepsilon} \right) \stackrel{\eqref{eqn:convex}}{\ge} 0.
$$
This yields $I_5 \leq 0$ and hence concludes the proof of~\eqref{e:fd}. 
\section{Rigorous proof of Theorem~\ref{t:main}} \label{s:complete}
In this section we provide the complete proof of~\eqref{e:tv}. In \S\ref{ss:data} we establish it under some further assumptions on the data, which are then removed through a fairly standard approximation argument in 
\S\ref{ss:app}.
\subsection{Proof under further assumptions on the data}\label{ss:data} 
Since $\mathrm{TotVar} \, u_0< + \infty$, then there are $u_-, u_+ \in [0, 1]$ such that 
\begin{equation} \label{e:limiti}
    \lim_{x \to - \infty} u_0(x) = u_-, \qquad \lim_{x \to + \infty} u_0(x) = u_+. 
\end{equation}
In this paragraph we establish~\eqref{e:tv} under the further assumptions 
\begin{equation}
  \label{e:hpkernel}
    \eta \in C^2 (]-\infty, 0[), \qquad \eta'' \in L^1 (]-\infty, 0[)
  \end{equation}
and 
\begin{equation}
\label{e:hpdato}
     u_0 \in \mathrm{Lip}(\mathbb R), \qquad u_0 (x) = u_+ \; \text{for every $x\ge R$,}
\end{equation}  
 for some suitable $R>0$.  
We now fix $T>0$ once and for all and establish the inequality in~\eqref{e:tv} for every $t \in [0, T]$. By the arbitrariness of $T$, this implies that~\eqref{e:tv} holds for every $t>0$.  We proceed according to the following steps. \\
 {\sc Step 1:} we point out that~\eqref{e:hpdato} yields 
 \begin{equation}
 \label{e:asintoto}
         w_\varepsilon (t, x) = u_+ \; \text{for every $t>0$ and $x\ge R + V(u_+) t$}.
 \end{equation} 
 Very loosely speaking, this is due to the fact that, by the kernel anisotropy, along the characteristic line $X_\varepsilon (\cdot, R, 0)$ the convolution term only ``sees" the value $u_+$, which in turn owing to~\eqref{e:X} implies that the speed of this characteristic line is exactly $V(u_+)$. 
 Property~\eqref{e:asintoto} can then be rigorously established by following the same argument sketched after the statement of~\cite[Lemma 13]{ColomboCrippaMarconiSpinolo}. Once can also show that 
\begin{equation}
\label{e:limitew}
      \lim_{x \to - \infty} w_\varepsilon (t, x) = u_- \quad \text{for every $t \in [0, T]$},
\end{equation}
where $u_-$ is the same as in~\eqref{e:limiti}. \\
 {\sc Step 2:} we recall that owing to~\eqref{e:V} $V(0)$ is the highest possible value attained by $V(w_\varepsilon)$, we fix $N \in \mathbb N$ such that $N \ge R$ and set 
 \begin{equation} 
\label{e:TVN} 
\begin{split}
      & \mathrm{Tot Var}_{N}  w_\varepsilon(t,\cdot):= \\
       & 
      \sup \! \left\{  \! |w_\varepsilon(t,x_{1}) \! - u_-\! |\! + \!
      \sum_{i=1}^{N-1} \! \! |w_\varepsilon(t,x_{i+1}) \! -\! w_\varepsilon(t,x_i) \!|\!:  X_\varepsilon(t, -N, 0) \! \leq \! x_1 \!\leq\! \dots \! \leq x_N\!= \!N \!+ \! V(u_+)t  \right\} \!, \\
         \end{split}
\end{equation}
where $X_\varepsilon$ is the same as in~\eqref{e:X} and, owing to~\eqref{e:asintoto}, $N + V(u_+) t = X_\varepsilon (t, N, 0)$. 
Owing to~\eqref{e:limitew} we have 
\begin{equation} \label{e:suptvN}
     \mathrm{Tot Var} \, w_\varepsilon(t,\cdot)= \lim_{N \to + \infty} \mathrm{Tot Var}_{N} w_\varepsilon(t,\cdot) = \sup_{N \in \mathbb N} \mathrm{Tot Var}_{N} w_\varepsilon(t,\cdot). 
\end{equation}
By {\sc Step 2} in the proof of~\cite[Theorem 1.1]{ColomboCrippaMarconiSpinolo2}, we have $ \mathrm{Tot Var} \, w_\varepsilon(t,\cdot) < + \infty$ for every $t\ge0$, and owing to~\eqref{e:suptvN} this implies $ \mathrm{Tot Var}_{N} w_\varepsilon(t,\cdot) < + \infty$ for every $t \ge 0$. We now establish some further properties of $ \mathrm{Tot Var}_{N} w_\varepsilon(t, \cdot)$. \\
{\sc Step 2A:} we show that the sup in~\eqref{e:TVN} is attained, i.e. for every $t\in [0, T]$ and every $N \ge R$ there are $N$ points $x_1\leq  x_2 \dots \leq  x_N$ such that 
\begin{equation}\label{e:max}
    \mathrm{Tot Var}_{N} w_\varepsilon(t,\cdot) =   |w_\varepsilon(t,x_{1}) - u_- | + \sum_{i=1}^{N-1} |w_\varepsilon(t,x_{i+1}) - w_\varepsilon(t,x_i)|. 
\end{equation}
Fix $t \in [0, T]$, $N \ge R$ and a maximizing sequence $x_{1k}, \dots, x_{N k}$ for $ |w_\varepsilon(t,x_{1}) - u_- | + \sum_{i=1}^{N -1} |w_\varepsilon(t,x_{i+1}) - w_\varepsilon(t,x_i)| $. We now point out that the points are all confined in the compact set $[ X_\varepsilon (t, -N, 0), N + V(u_+)t],$ and conclude that up to subsequences the maximizing sequences converge to some limit points $x_{1}, \dots, x_{N}$ satisfying~\eqref{e:max}. \\
 {\sc Step 2B:} we show that the map $t \mapsto    \mathrm{Tot Var}_{N} w_\varepsilon(t,\cdot)$ is Lipschitz continuous, and henceforth a.e. differentiable. 
To this end we fix $t_1, t_2 \in [0, T]$ and assume without loss of generality that $ \mathrm{Tot Var}_{ N} w_\varepsilon(t_1,\cdot) \leq \mathrm{Tot Var}_{ N} w_\varepsilon(t_2,\cdot).$ We rely on {\sc Step 2A} and find $x_1, \dots, x_N$, such that~\eqref{e:max} holds at $t =t_1$. By recalling definition~\eqref{e:TVN} we infer that 
\begin{equation*}
\begin{split}
     \mathrm{Tot Var}_{N}  w_\varepsilon(t_2,\cdot) \leq  & |w_\varepsilon(t_2, X_\varepsilon (t_2, x_{1}, t_1)) - u_- | \\ & + \sum_{i=1}^{N-1} |w_\varepsilon(t_2, X_\varepsilon (t_2, x_{i+1}, t_1)) - w_\varepsilon(t_2, X_\varepsilon (t_2, x_{i}, t_1))|,
\end{split}
\end{equation*}
which implies that 
\begin{equation*}\begin{split}
   & |  \mathrm{Tot Var}_{N} w_\varepsilon(t_2,\cdot) -  \mathrm{Tot Var}_{N} w_\varepsilon(t_1,\cdot) | = 
   \mathrm{Tot Var}_{N} w_\varepsilon(t_2,\cdot) -  \mathrm{Tot Var}_{ N} w_\varepsilon(t_1,\cdot) \\
     &   \leq      |w_\varepsilon(t_2, X_\varepsilon (t_2, x_{1}, t_1)) - u_- | -     |w_\varepsilon(t_1,x_{1}) - u_- | 
     \phantom{\sum_{i=1}^N}\\ & \qquad + 
   \sum_{i=1}^{N-1} |w_\varepsilon(t_2,X_\varepsilon (t_2, x_{i+1}, t_1)) - w_\varepsilon(t_2,X_\varepsilon (t_2, x_{i}, t_1))| - |w_\varepsilon(t_1,x_{i+1}) - w_\varepsilon(t_1,x_i)| \\
   &
   \leq   |w_\varepsilon(t_2,  X_\varepsilon (t_2, x_{1}, t_1) ) - w_\varepsilon (t_1, x_1) | + \sum_{i=1}^{N-1}  |w_\varepsilon(t_2,X_\varepsilon (t_2, x_{i+1}, t_1)) - w_\varepsilon(t_1,x_{i+1})| \\
    & \qquad +  \sum_{i=1}^{N-1}
    |w_\varepsilon(t_2,X_\varepsilon (t_2, x_{i}, t_1)) - w_\varepsilon(t_1,x_i)|.
   \end{split}
\end{equation*}
Since the functions $X_\varepsilon$ and $w_\varepsilon$ are both Lipschitz continuous  this concludes {\sc Step 2C}. \\
{\sc Step 3:}  we now fix $N \ge R$ and $t_\ast \in ]0, T[$ such that the map $t \mapsto  \mathrm{Tot Var}_{N} w_\varepsilon(t,\cdot) $ is differentiable at $t_\ast$. Assume for a moment that we have shown that 
\begin{equation} \label{e:dertvN}
    \left. \frac{d}{dt}  \mathrm{Tot Var}_{N} w_\varepsilon(t,\cdot) \right|_{t = t_\ast} \leq 0. 
\end{equation}
Owing to the arbitrariness of $t_\ast$ by integrating in time we get 
$$
    \mathrm{Tot Var}_{N} w_\varepsilon(t,\cdot) \leq   \mathrm{Tot Var}_{N} w_\varepsilon(0,\cdot) \stackrel{\eqref{e:suptvN}}{\leq} 
    \mathrm{Tot Var} \, w_\varepsilon(0,\cdot) 
$$
and owing to~\eqref{e:suptvN} this yields~\eqref{e:tv}. To establish~\eqref{e:dertvN} we set 
\begin{equation} \label{e:g2}
   g(t) : =  \mathrm{Tot Var}_{N} w_\varepsilon(t,\cdot)
\end{equation}
and point out that, owing to {\sc Step 2B}, the function $g:[0, T  ] \to \mathbb R_+$ is Lipschitz continuous.  
Next, we fix $t_\ast$ at which $g$ is differentiable, recall {\sc Step 2A} and fix $x_1, \dots, x_N = N + V(u_+)t_\ast$ such that 
\begin{equation} \label{e:tastmax}
    \mathrm{Tot Var}_{N} w_\varepsilon(t_\ast,\cdot) =  |w_\varepsilon(t_\ast,x_{1}) - u_-| + \sum_{i=1}^{N-1} |w_\varepsilon(t_\ast,x_{i+1}) - w_\varepsilon(t_\ast,x_i)|. 
\end{equation}
Up to removing some intermediate points and relabeling the points (if needed) 
we find $m \leq N$ such that $x_1 < x_2<\dots < x_m=N+ V(u_+)t$ are all distinct points and one of the following two cases is verified:
\begin{itemize}
\item[i)] either  
\begin{equation} \label{e:inmezzo}
    w_\varepsilon (t_\ast, x_1) \ge u_-,  \quad w_\varepsilon (t_\ast, x_i) \ge w_\varepsilon(t_\ast, y) \; \text{for every $y \in [x_i, x_{i+1}]$, $i=1, \dots, m-1$ odd}
\end{equation}
and 
\begin{equation} \label{e:inmezzo2}
 w_\varepsilon (t_\ast, x_i) \leq w_\varepsilon(t_\ast, y) \; \text{for every $y \in [x_i, x_{i+1}]$ and $i=2, \dots, m-1$ even},
\end{equation}
\item[ii)] or 
$$
    w_\varepsilon (t_\ast, x_1) \leq u_-,  \quad w_\varepsilon (t_\ast, x_i) \leq w_\varepsilon(t_\ast, y) \; \text{for every $y \in [x_i, x_{i+1}]$, $i=1, \dots, m-1$ odd}
$$
and 
$$
 w_\varepsilon (t_\ast, x_i) \ge w_\varepsilon(t_\ast, y) \; \text{for every $y \in [x_i, x_{i+1}]$ and $i=2, \dots, m-1$ even},
$$
\end{itemize}
We now focus on case i) above. The proof in case ii) is similar. Note that in case i), recalling that $w_\varepsilon (t_\ast, x_m) = u_+$, we have 
\begin{equation} \label{e:moddeven}
\begin{split}
  g(t_\ast) \stackrel{\eqref{e:g2}}{=} \mathrm{Tot Var}_{N} w_\varepsilon(t_\ast,\cdot) &=
   w_\varepsilon(t_\ast,x_{1}) - u_- + \sum_{i=1}^{m-1} |w_\varepsilon(t_\ast,x_{i+1}) - w_\varepsilon(t_\ast,x_i)| \\
   &   = - u_-  + \left\{
   \begin{array}{ll}
   - u_+ + 2 \displaystyle{ \sum_{i=1}^{m-1} (-1)^{i+1}  w_\varepsilon (t_\ast, x_i) } & \text{$m$ even} \\
    u_+ + 2 \displaystyle{ \sum_{i=1}^{m-1} (-1)^{i+1}  w_\varepsilon (t_\ast, x_i)}  & \text{$m$ odd.} \\
\end{array}
 \right.\\
\end{split}
\end{equation}
We recall the notation~\eqref{e:X} and set 
\begin{equation} \label{e:effe}
    f(t) : = - u_-  + \left\{
   \begin{array}{ll}
   - u_+ + 2 \displaystyle{ \sum_{i=1}^{m-1} (-1)^{i+1}  w_\varepsilon (t_\ast, X_\varepsilon(t, x_i, t_\ast)) } & \text{$m$ even} \\
    u_+ + 2 \displaystyle{ \sum_{i=1}^{m-1} (-1)^{i+1}  w_\varepsilon (t_\ast, X_\varepsilon(t, x_i, t_\ast))}  & \text{$m$ odd.} \\
\end{array}
 \right.\\
\end{equation}
Note that $f(t) \leq g(t)$ for every $t \in [0, T]$, where $g$ is the same as in~\eqref{e:g2}.  By recalling~\eqref{e:moddeven} and Lemma~\ref{l:fg} we conclude that to establish~\eqref{e:dertvN} it suffices to show that  $f'(t_\ast) \leq 0$. \\
{\sc Step 4:} by combining~\eqref{e:ewgen} and~\eqref{e:effe} we get 
\begin{equation} \label{e:tesi}
\begin{split}
f'(t_\ast) & = 2  \sum_{i=1}^{m-1}   \frac{(-1)^{i+1} }{\varepsilon^2} \int_{x_i}^{+\infty} \eta' \left( \frac{x-y}{\varepsilon} \right) [V(w_\varepsilon (t_\ast, x_i)) - V(w_\varepsilon (t, y) )] u_\varepsilon (t_\ast, y) dy \\
& =
     \frac{2}{\varepsilon^2} \sum_{i=1}^{m-1} \int_{x_i}^{x_{i+1}} \sigma_{ i} (t_\ast, y) u_\varepsilon ( t_\ast, y) dy +  \frac{2}{\varepsilon^2}
      \int_{x_m}^{+\infty} \sigma_m (t_\ast, y)u_\varepsilon (t_\ast, y) dy \\
     \end{split}
\end{equation}
where 
\begin{equation}
\label{e:sigma1}
       \sigma_1 (t_\ast, y) =   
       \eta' \left( \frac{x_1 -y}{\varepsilon}\right)   [V(w_\varepsilon(t_\ast, x_1) ) - V(w_\varepsilon( t_\ast, y)) ]  
\end{equation}
and 
\begin{equation}
\begin{split}
\label{e:sigma}
       \sigma_i (t_\ast, y) & =   
       \sum_{j=1}^{i-1} 
     \left[ \sum_{k=1}^j (-1)^{k+1} 
       \eta' \left( \frac{x_k -y}{\varepsilon}\right)\right] [V(w_\varepsilon(t_\ast, x_j) - V(w_\varepsilon(t_\ast, x_{j+1})]\\ &
      \qquad +
       \left[ \sum_{k=1}^i (-1)^{k+1} \eta' \left( \frac{x_k -y}{\varepsilon}\right)  \right] [V(w_\varepsilon(t_\ast, x_i) ) - V(w_\varepsilon(t_\ast, y)) ]   
        \quad \text{$i\ge2$}.
\end{split}
\end{equation}
To verify the above formula we rely on an induction argument. By a straightforward computation,  the formula is verified for $i=2$. Let us now assume that it holds at the $i$-th step, and  check that it is also verified at the $(i+1)$-th.
We have 
\begin{equation*}
\begin{split}
      \sigma_{i +1}(t_\ast, y) & =  \sigma_i (t_\ast,  y) + (-1)^{i+2}  \eta' \left( 
      \frac{x_{i+1} -y}{\varepsilon}\right)   [V(w_\varepsilon(t_\ast, x_{i+1}) ) - V(w_\varepsilon(t_\ast, y)) ]  \\ & 
      \stackrel{\eqref{e:sigma}}{=}    
       \sum_{j=1}^{i-1} 
     \left[ \sum_{k=1}^j (-1)^{k+1} 
       \eta' \left( \frac{x_k -y}{\varepsilon}\right)\right] [V(w_\varepsilon(t_\ast, x_j) - V(w_\varepsilon(t_\ast, x_{j+1})]\\ &
      \qquad +
       \left[ \sum_{k=1}^i (-1)^{k+1} \eta' \left( \frac{x_k -y}{\varepsilon}\right)  \right] [V(w_\varepsilon(t_\ast, x_i) ) - V(w_\varepsilon(t_\ast, y)) ]   \\
& \qquad +  (-1)^{i+2}  \eta' \left( 
      \frac{x_{i+1} -y}{\varepsilon}\right)   [V(w_\varepsilon(t_\ast, x_{i+1}) ) - V(w_\varepsilon(t_\ast, y)) ]
      \end{split}
\end{equation*}
and hence 
\begin{equation*}     \begin{split}
    \sigma_{i +1}(t_\ast, y) & =   \sum_{j=1}^{i-1} 
     \left[ \sum_{k=1}^j (-1)^{k+1} 
       \eta' \left( \frac{x_k -y}{\varepsilon}\right)\right] [V(w_\varepsilon(t_\ast, x_j) - V(w_\varepsilon(t_\ast, x_{j+1})]\\ &
      \qquad +
       \left[ \sum_{k=1}^i (-1)^{k+1} \eta' \left( \frac{x_k -y}{\varepsilon}\right)  \right] [V(w_\varepsilon(t_\ast, x_i) ) - V(w_\varepsilon(t_\ast, x_{i+1})) ] \\  &\qquad +   \sum_{k=1}^{i+1} (-1)^{k+1} \eta' \left( \frac{x_k -y}{\varepsilon}\right)   [V(w_\varepsilon(t_\ast, x_{i+1}) ) - V(w_\varepsilon(t_\ast, y)) ] ,
\end{split}
\end{equation*}
which is consistent with~\eqref{e:sigma}. \\
{\sc Step 5:} we  recall that $u_\varepsilon \ge 0$ owing to~\eqref{e:mp}, and by plugging this inequality in~\eqref{e:tesi} we conclude that to show that $f'(t_\ast) \leq 0$ it suffices to show that   
\begin{equation}\label{e:sigmai}
       \sigma_i (t_\ast, \cdot) \leq 0 \;  \text{a.e. on $]x_i, x_{i+1}[$ for $i=1, \dots, m-1$}, \qquad \sigma_m(t_\ast, \cdot) \leq 0 \; \text{on $]x_m, +\infty[$ }. 
\end{equation}       
To establish~\eqref{e:sigmai} we first focus on $i=1$. Note that~\eqref{e:inmezzo} yields $w_\varepsilon(t_\ast, y) \leq 
w_\varepsilon(t_\ast, x_1)$ for every $y \in [x_1, x_2]$ and by using the inequality $V' \leq 0$ we conclude that $V(w_\varepsilon(t_\ast, x_1)) \leq 
V(w_\varepsilon(t_\ast, y))$  for every $y \in [x_1, x_2]$. By recalling the inequality $\eta'\ge 0$ we conclude that $\sigma_1(t_\ast, \cdot) \leq 0$ on $[x_1, x_2]$.

We now consider the case $i>1$. Assume for a moment that we have established the inequalities 
\begin{equation}\label{e:j1} 
 \sum_{k=1}^j (-1)^{k+1} 
       \eta' \left( \frac{x_k -y}{\varepsilon}\right) \ge 0 \quad  \text{if $j=1, \dots, m$, $j$ odd}
\end{equation}  
and
\begin{equation}       \label{e:j2}
        \sum_{k=1}^j (-1)^{k+1}  \eta' \left( \frac{x_k -y}{\varepsilon}\right)
               \leq 0 \quad \text{if $j=2, \dots, m$, $j$ even.} 
\end{equation}
By combining~\eqref{e:j1} with~\eqref{e:inmezzo},~\eqref{e:j2} with~\eqref{e:inmezzo2} and recalling the inequality $V' \leq 0$ we get that the first and (choosing $j=i$) the second term in~\eqref{e:sigma} are nonpositive. Hence, to establish~\eqref{e:sigmai} we are left to establish~\eqref{e:j1} and~\eqref{e:j2}. 
We first establish~\eqref{e:j1}. If $j=1$, then 
\begin{equation} \label{e:jodd1}
     \sum_{k=1}^1 (-1)^{k+1} 
       \eta' \left( \frac{x_k -y}{\varepsilon}\right) =  \eta' \left( \frac{x_1 -y}{\varepsilon}\right) \stackrel{\eqref{e:eta}}{\ge} 0.  
\end{equation}
If $j\ge 3$ is odd, then we use the equality 
\begin{equation}\label{e:jodd2} \begin{split}
      \sum_{k=1}^j (-1)^{k+1}  \eta' \left( \frac{x_k -y}{\varepsilon}\right) =  
       \eta' \left( \frac{x_1 -y}{\varepsilon}\right)  &+ 
       \sum_{h=2, \, h \; \text{even}}^{j-1}    \left[ \eta' \left( \frac{x_{h+1} -y}{\varepsilon}\right)    - \eta' \left( \frac{x_{h} -y}{\varepsilon}\right) \right] \\
       & \text{for every $j=3, \dots, m$, $j$ odd}.
       \end{split}
\end{equation}
Since $\eta' \ge 0$ and $\eta'$ is a  a non-decreasing function, each of the terms in the above sum is non-negative, which combined with~\eqref{e:jodd1} establishes~\eqref{e:j1}. We now turn to~\eqref{e:j2}. We have 
$$
      \sum_{k=1}^j (-1)^{k+1}  \eta' \! \left( \frac{x_k -y}{\varepsilon}\right) \! = \! \!  \! \!  \!  \!
       \sum_{h=1, \, h \; \text{odd}}^{j-1}   \left[ \eta' \! \left(\! \frac{x_h -y}{\varepsilon}\right)  
      \!  - \! \eta'\! \left( \frac{x_{h+1} -y}{\varepsilon}\right)\! \right]  \, \text{$\forall \, j=2, \dots, m$, $j$ even}.
$$
Owing to the assumption that $\eta'$ is a non-decreasing function, each of the terms in the above sum is non-positive and this establishes~\eqref{e:j2}.  This concludes the proof of the inequality $f'(t_\ast) \leq 0$ and hence of~\eqref{e:tv}. 
\subsection{Conclusion of the proof of~\eqref{e:tv}}\label{ss:app}
To complete the proof of~\eqref{e:tv} we are left to remove the assumptions~\eqref{e:hpkernel} and~\eqref{e:hpdato}. We can use the same approximation argument as in {\sc Step 3} of the proof of~\cite[Theorem 1.1]{ColomboCrippaMarconiSpinolo2}, the only difference is that we also have to get the second condition in~\eqref{e:hpdato}. To achieve this, we can define $u_{0n}$ as the convolution of the compactly supported kernel $\rho_n$ with 
$$
    \widetilde u_{0n} : = 
    \left\{
    \begin{array}{ll}
    u_0(x) & x \leq n \\
    u_+ & x \ge n \\  
    \end{array}
    \right.
$$
In this way~\eqref{e:hpdato} is verified by $u_{0n}$ for some $R$ depending on $n$ and on the size of the support of $\rho_n$. 

\section*{Acknowledgments}
 This note is based on a talk given by L.V.S. at the ``Journ\'ees des EDP" 2023. L.V.S. warmly thanks the organizers for the kind invitation. G.C.  is supported by SNF Project 212573 FLUTURA-Fluids, Turbulence, Advection. E.M. and L.V.S. are members of the PRIN 2022 PNRR Project P2022XJ9SX and of the GNAMPA group of INDAM. E.M. is supported by the European Union Horizon 20 research and innovation program under the Marie Sklodowska-Curie grant No. 101025032. LVS is a member of the PRIN 2020 Project 20204NT8W4, of the PRIN 2022 Project 2022YXWSLR, and of the CNR FOE 2022 Project STRIVE.

% The next command determines the bibliography style. Please do not
% change this.

\bibliographystyle{plain}

%  This inserts the bib file, biblio.bib
\bibliography{jedp}

\begin{thebibliography}{10}

\bibitem{ACT}
P.~Amorim, R.~M. Colombo, and A.~Teixeira.
\newblock On the numerical integration of scalar nonlocal conservation laws.
\newblock {\em ESAIM Math. Model. Numer. Anal.}, 49(1):19--37, 2015.

\bibitem{Sedimentation}
F.~Betancourt, R.~B\"urger, K.~H. Karlsen, and E.~M. Tory.
\newblock On nonlocal conservation laws modelling sedimentation.
\newblock {\em Nonlinearity}, 24(3):855--885, 2011.

\bibitem{BlandinGoatin}
S.~Blandin and P.~Goatin.
\newblock Well-posedness of a conservation law with non-local flux arising in
  traffic flow modeling.
\newblock {\em Numer. Math.}, 132(2):217--241, 2016.

\bibitem{BressanShen}
A.~Bressan and W.~Shen.
\newblock On traffic flow with nonlocal flux: a relaxation representation.
\newblock {\em Arch. Ration. Mech. Anal.}, 237(3):1213--1236, 2020.

\bibitem{BressanShen2}
A.~Bressan and W.~Shen.
\newblock Entropy admissibility of the limit solution for a nonlocal model of
  traffic flow.
\newblock {\em Commun. Math. Sci.}, 19(5):1447--1450, 2021.

\bibitem{CP}
P.~Calderoni and M.~Pulvirenti.
\newblock Propagation of chaos for {B}urgers' equation.
\newblock {\em Ann. Inst. H. Poincar\'{e} Sect. A (N.S.)}, 39(1):85--97, 1983.

\bibitem{Chiarello}
F.~A. Chiarello.
\newblock An overview of non-local traffic flow models.
\newblock In {\em Mathematical descriptions of traffic flow: micro, macro and
  kinetic models. Selected papers based on the presentations of the
  mini-symposium at ICIAM 2019, Valencia, Spain, July 2019}, pages 79--91.
  Cham: Springer, 2021.

\bibitem{OttoAutori}
G.M. Coclite, M~Colombo, G~Crippa, N.~De~Nitti, A.~Keimer, E.~Marconi,
  L.~Pflug, and L.V. Spinolo.
\newblock Ole{\u\i}nik-type estimates for nonlocal conservation laws and
  applications to the nonlocal-to-local limit.
\newblock {\em Preprint ArXiv:2304.01309}, 2023.

\bibitem{CocliteCoronDNKeimerPflug}
G.M. Coclite, J.-M. Coron, N.~De~Nitti, A.~Keimer, and L.~Pflug.
\newblock A general result on the approximation of local conservation laws by
  nonlocal conservation laws: The singular limit problem for exponential
  kernels.
\newblock {\em Ann. Inst. H. Poincar\'{e} C Anal. Non Lin\'{e}aire}, 2022.

\bibitem{CDNKP}
G.M. Coclite, N.~De~Nitti, A.~Keimer, and L.~Pflug.
\newblock Singular limits with vanishing viscosity for nonlocal conservation
  laws.
\newblock {\em Nonlinear Anal.}, 211, 2021.

\bibitem{CKR}
G.M. Coclite, K.H. Karlsen, and N.H. Risebro.
\newblock A nonlocal lagrangian traffic flow model and the zero-filter limit.
\newblock {\em Preprint ArXiv:2302.03889}, 2023.

\bibitem{ColomboCrippaGraffSpinolo}
M.~Colombo, G.~Crippa, M.~Graff, and L.~V. Spinolo.
\newblock On the role of numerical viscosity in the study of the local limit of
  nonlocal conservation laws.
\newblock {\em ESAIM, Math. Model. Numer. Anal.}, 55(6):2705--2723, 2021.

\bibitem{ColomboCrippaMarconiSpinolo}
M.~Colombo, G.~Crippa, E.~Marconi, and L.~V. Spinolo.
\newblock Local limit of nonlocal traffic models: convergence results and total
  variation blow-up.
\newblock {\em Ann. Inst. H. Poincar\'{e} C Anal. Non Lin\'{e}aire},
  38(5):1653--1666, 2021.

\bibitem{ColomboCrippaMarconiSpinolo2}
M.~Colombo, G.~Crippa, E.~Marconi, and L.~V. Spinolo.
\newblock Nonlocal traffic models with general kernels: singular limit, entropy
  admissibility, and convergence rate.
\newblock {\em Arch. Ration. Mech. Anal.}, 247(2):32, 2023.

\bibitem{ColomboCrippaSpinolo}
M.~Colombo, G.~Crippa, and L.~V. Spinolo.
\newblock On the singular local limit for conservation laws with nonlocal
  fluxes.
\newblock {\em Arch. Rat. Mech. Anal.}, 233(3):1131--1167, 2019.

\bibitem{ColomboGaravelloMercier}
R.~M. Colombo, M.~Garavello, and M.~L\'ecureux-Mercier.
\newblock A class of nonlocal models for pedestrian traffic.
\newblock {\em Math. Models Methods Appl. Sci.}, 22(4):1150023, 34, 2012.

\bibitem{ColomboHertyMercier}
R.~M. Colombo, M.~Herty, and M.~Mercier.
\newblock Control of the continuity equation with a non local flow.
\newblock {\em ESAIM Control Optim. Calc. Var.}, 17(2):353--379, 2011.

\bibitem{CLM}
G.~Crippa and M.~L{\'e}cureux-Mercier.
\newblock Existence and uniqueness of measure solutions for a system of
  continuity equations with non-local flow.
\newblock {\em NoDEA Nonlinear Differential Equations Appl.}, 20(3):523--537,
  2013.

\bibitem{Dafermos_book}
C.~M. Dafermos.
\newblock {\em Hyperbolic conservation laws in continuum physics}, volume 325
  of {\em Grundlehren der Mathematischen Wissenschaften [Fundamental Principles
  of Mathematical Sciences]}.
\newblock Springer-Verlag, Berlin, fourth edition, 2016.

\bibitem{DHSS}
Q.~Du, K.~Huang, J.~Scott, and W.~Shen.
\newblock A space-time nonlocal traffic flow model: relaxation representation
  and local limit.
\newblock {\em Discrete Contin. Dyn. Syst.}, 43(9):3456--3484, 2023.

\bibitem{FGKP}
J.~Friedrich, S.~G\"ottlich, A.~Keimer, and L.~Pflug.
\newblock Conservation laws with nonlocal velocity - the singular limit
  problem.
\newblock {\em Preprint ArXiv:2210.12141}, 2022.

\bibitem{KeimerPflug}
A.~Keimer and L.~Pflug.
\newblock Existence, uniqueness and regularity results on nonlocal balance
  laws.
\newblock {\em J. Differential Equations}, 263(7):4023--4069, 2017.

\bibitem{KeimerPflug2}
A.~Keimer and L.~Pflug.
\newblock On approximation of local conservation laws by nonlocal conservation
  laws.
\newblock {\em J. Math. Anal. Appl.}, 475(2):1927--1955, 2019.

\bibitem{KPpreprint}
A.~Keimer and L.~Pflug.
\newblock On the singular limit problem for nonlocal conservation laws: A
  general approximation result for kernels with fixed support.
\newblock {\em Preprint ArXiv:2310.09041}, 2023.

\bibitem{Kruzkov}
S.~N. Kru{\v{z}}kov.
\newblock First order quasilinear equations with several independent variables.
\newblock {\em Mat. Sb. (N.S.)}, 81 (123):228--255, 1970.

\bibitem{Kuznetsov}
N.~N. Kuznetsov.
\newblock Accuracy of some approximate methods for computing the weak solutions
  of a first-order quasi-linear equation.
\newblock {\em U.S.S.R. Comput. Math. Math. Phys.}, 16(6):105--119, 1978.

\bibitem{LW}
M.~Lighthill and G.~Whitham.
\newblock On kinematic waves. {II. A} theory of traffic flow on long crowded
  roads.
\newblock {\em Proceedings of the Royal Society of London: Series A.},
  229:317--345, 1955.

\bibitem{R}
P.~I. Richards.
\newblock Shock waves on the highway.
\newblock {\em Operations Res.}, 4:42--51, 1956.

\bibitem{Zumbrun}
K.~Zumbrun.
\newblock On a nonlocal dispersive equation modeling particle suspensions.
\newblock {\em Quart. Appl. Math.}, 57(3):573--600, 1999.

\end{thebibliography}

% This command signals the end of the file.
\end{document}